\documentclass{birkjour}
\usepackage{amssymb,amsfonts,amsmath,enumerate,color,url,hyperref,mathbbol}
\usepackage{tikz}

 \newtheorem{thm}{Theorem}[section]
 \newtheorem{cor}[thm]{Corollary}
 \newtheorem{lem}[thm]{Lemma}
 \newtheorem{prop}[thm]{Proposition}
 \theoremstyle{definition}
 \newtheorem{defn}[thm]{Definition}
 \theoremstyle{remark}
 \newtheorem{rem}[thm]{Remark}
 \newtheorem*{ex}{Example}
 \numberwithin{equation}{section}

\begin{document}

\title[Moments-preserving cosines and semigroups]
 {On moments-preserving cosine families and semigroups in $C[0,1]$}
\author[Adam Bobrowski]{Adam Bobrowski}

\address{%
Institute of Mathematics,\\
Polish Academy of Sciences,\\
\'Sniadeckich 8, \\
00-956 Warsaw, Poland \\
\emph{on leave from}
\\
Lublin University of Technology\\
Nadbystrzycka 38A\\
20-618 Lublin, Poland}

\email{a.bobrowski@pollub.pl}

\author[Delio Mugnolo]{Delio Mugnolo}

\address{%
Institut f\"ur Analysis\\
Universit\"at Ulm\\
89069 Ulm\\
Germany}

\email{delio.mugnolo@uni-ulm.de}

\newcommand{\cxi}{(\xi_i)_{i\in \N} }
\newcommand{\lam}{\lambda}
\newcommand{\eps}{\epsilon}
\newcommand{\ud}{\, \mathrm{d}}
\newcommand{\pr}{\mathbb{P}}
\newcommand{\f}{\mathcal{F}}
\newcommand{\s}{\mathcal{S}}
\newcommand{\h}{\mathcal{H}}
\newcommand{\ai}{\mathcal{I}}
\newcommand{\R}{\mathbb{R}}
\newcommand{\C}{\mathbb{C}}
\newcommand{\Z}{\mathbb{Z}}
\newcommand{\N}{\mathbb{N}}
\newcommand{\e}{\mathrm {e}}
\newcommand{\tif}{\tilde {f}}
\newcommand{\slam}{\sqrt {\lam}}
\newcommand{\Id}{{\rm Id}}
\newcommand{\cic}{C_{\rm mp}}
\newcommand{\cod}{C_{\text{\rm odd}}[0,1]}
\newcommand{\cev}{C_{\text{\rm even}}[0,1]}
\newcommand{\cez}{C_0(0,1]}
\newcommand{\fod}{f_{\text{\rm odd}}} 
\newcommand{\fev}{f_{\text{\rm even}}} 
\newcommand {\sem}[1]{\mbox{$\left (\e^{t{#1}}\right )_{t \ge 0}$}}
\newcommand{\tr}{\textcolor{red}}

\thanks{Version of \today}


\dedicatory{}

\begin{abstract}
We use the newly developed Kelvin's method of images \cite{kosinusy,kelvin} to show existence of a unique cosine family generated by a restriction of the Laplace operator in $C[0,1]$, that preserves the first two moments. We characterize the domain of its generator by specifying its boundary conditions. Also, we show that it enjoys inherent symmetry properties, and in particular that it leaves the subspaces of odd and even functions invariant. Furthermore, we provide information on long-time behavior of the related semigroup.
\end{abstract}

\maketitle

\section{Introduction}

While evolution equations on domains are usually equipped with boundary conditions, these can be partially or completely replaced by integral conditions {almost without changes when it comes to relevant issues like well-posedness and spectral asymptotics}. This has been first observed by J.R. Cannon 50 years ago \cite{Can63}, his seminal investigations having been generalized and applied by several authors ever since. Depending on the considered model, certain integral conditions have in fact a clear physical meaning (like conservation of mass) which can actually be expected by real-world systems.


A semigroup-theoreti\-cal study of diffusion and wave equations associated with one-dimensional Laplace operators equipped with integral conditions has recently been commenced in \cite{mugnic}, where an abstract framework for studying such problems has been proposed. In particular, it was shown there that the requirement that the first two moments (i.e., both moments of order $0$ and $1$) vanish, leads to well-posed wave and diffusion equations in the space of $H^{-1}(T)$-distributions of zero average, where $H^{-1}(T)$ is the dual space of $H^1(T):=\{f\in H^1(0,1):f(0)=f(1)\}$. Moreover, spectral properties and long time behavior of the related solutions has been studied.  

In this paper, we present an alternative approach to such problems. Namely, using Lord Kelvin's method of images, shown recently to be a useful tool for proving generation theorems in \cite{kosinusy,kelvin}, we construct, in a quite explicit way, a cosine family in $C[0,1]$, generated by a restriction of the Laplace operator and preserving the first two moments. As it turns out, the domain of this cosine family's generator {is the space} of twice continuously differentiable functions $f\in C[0,1]$ satisfying the boundary conditions 
\begin{equation} f'(0) = f(1) - f(0) = f'(1).\label{bc} \end{equation}
Moreover, we show that  among the semigroups generated by various realizations of the Laplace operator in $C[0,1]$ there is only one that preserves the first two moments: this is the semigroup with boundary conditions \eqref{bc}.
Of course, by the Weierstra{ss} formula, this implies that the same is true for cosine families: among cosine families generated by one-dimensional Laplace operator in $C[0,1]$, there is but one that preserves the first two moments. 

These results, contained in the main Section \ref{sec:cfc01}, complement those of~\cite{kosinusy,ChiKeyWar07,warma}, where quite explicit formulae for the cosine family generated by the second derivative had been found in the case of local Robin and Wentzell-Robin boundary conditions. 
They may also be compared with an explicit solution found for a problem investigated in~\cite{LesEllIng98}, where other, related boundary conditions were discussed.  


One of the advantages of the Lord Kelvin's method of images is that often it provides an explicit form for the searched-for semigroup or cosine family. In our case this explicit form, referred to as the \emph{abstract Kelvin formula}, involves an extension $\tilde f$ of a member $f$ of $C[0,1]$ to the whole of $\R,$ see \eqref{kelvin}. Unfortunately, $\tilde f$  must be calculated iteratively, and no closed-form is available. Nevertheless, analysis of such extensions gives some insight into the nature of the moments-preserving cosine family, and exhibits its inherent symmetry properties. In Section \ref{s2}, we show that extensions of even (odd) functions about $\frac 12$ are even (odd). This implies that the cosine family leaves the subspaces of odd and even functions invariant. Interestingly, the moments-preserving cosine family, as restricted to the space of even functions is the same as the cosine family generated by the Laplace operator with Neumann boundary conditions. Moreover, `the odd part' of the cosine family is isometrically isomorphic to the cosine family with Robin boundary condition investigated previously in \cite{kelvin}.

Our final section is devoted to asymptotic behavior of the moments-preserving semigroup. It is clear from the Weierstrass formula that since the moments-preserving cosine family may be decomposed into its even and odd parts, the same is true for the semigroup. Additionally, the limit behavior of the even part, associated with the Neumann Laplace operator, is well known: in the limit, trajectories homogenize and become constant. The odd part is slightly more complicated: the related physical (or biological: see \cite{kelvin,kazlip}) process  involves particles diffusing freely in the open unit interval with constant inflow of particles through the boundary $x=1$ and outflow of particles through the boundary $x=0.$ This suggest that in the limit the distribution of particles should stabilize. This hypothesis can be proved by Hilbert space methods: The main result of Section \ref{s3} (see formula \eqref{glim}) states that as $t\to \infty$ the trajectory of the moments-preserving semigroup converges to a linear combination of $f_0=1_{[0,1]}$ (related to the even part) and of $f_1\in C[0,1]$, given by $f_1(x)=12x-6$, (related to the odd part) with coefficients in the combination being moments (about 0) of the trajectory's starting point.

\section{Moments-preserving cosine families in $C[0,1]$: a generation theorem}\label{sec:cfc01}

Let $C[0,1]$ be the Banach space of continuous functions on the unit interval, and $C(\R)$ be the Fr\'echet space of continuous functions on $\R$ with topology of almost uniform convergence. In what follows we think of real-valued functions, but this is merely to fix attention; the same analysis can be performed in the space of complex functions, as well. Let $(C(t))_{t \in \R}$ be the basic cosine family in $C(\R)$ given by the D'Alembert formula,
\begin{equation}
\label{dalemb}
 C(t) f (x) := \frac 12 (f(x+t) + f(x-t)) , \qquad t,x \in \R.
 \end{equation}
Also, let $F_i$ denote the moment of order $i$ about 0, i.e., let it be the linear functional on $C[0,1]$ defined by
\begin{equation} \label{fi} F_i f := \int_0^1 x^i f(x) 
\ud x,\qquad {i\in \mathbb N}.\end{equation} 
With an abuse of notation, we will denote by $F_i$ also the linear functional on $C(\R)$ defined by
\begin{equation} \label{fiR} F_i f := \int_0^1 x^i f(x) 
\ud x,\qquad  {i\in \mathbb N}.\end{equation}
Clearly, $F_i$ is continuous both on $C[0,1]$ and $C(\R)$ {for all $i\in \mathbb N$.}

In the theory of semigroups of linear operators and the related theory of cosine families, Lord Kelvin's method of images can be thought of as a way of constructing families of operators generated by an operator with a boundary condition by means of families generated by the same operator in a larger space, where no boundary conditions are imposed (cf.~\cite{kosinusy,kelvin}). In our particular context, the method boils down to constructing a cosine family
$\cic= (\cic (t))_{t \in \R}$ in $C[0,1]$ via the formula 
\begin{equation} \label{kelvin} \cic (t) f(x) = C(t) \tilde f(x), \qquad x \in [0,1],\; t\in \mathbb R,\; f \in C[0,1],\end{equation}  
where `mp' stands for `moments-preserving' and, more importantly, $\tif \in C(\R) $ is a certain extension of $f$, chosen in such a way that $(\cic(t))_{t\in \mathbb R}$ preserves both $F_0$ and $F_1$. To be more specific: Given $f\in C[0,1]$, 
we are looking for an $\tif:\mathbb R\to \mathbb R$ such that 
\begin{itemize} 
\item [(a) ] $\tif \in C(\R)$ and $ \tif (x) = f(x)$ for all $x \in [0,1]$,
\item [(b) ] $F_0 C(t) \tif = F_0 f$ for all $t \in \R$, and
\item [(c) ] $F_1 C(t) \tif  = F_1 f$ for all $t \in \R$.
\end{itemize}
Existence of such an extension is secured by Proposition \ref{lem_main}, later on. We note that if (b) holds, then (c) may be expressed equivalently as follows: $G C(t) \tif  = G f$ for all $t \in \R$,
where $Gf=G_af = \int_0^1 (a - x) f(x) \ud x$, and $a\in \R$ is fixed. In other words, preservation of the first two moments about $0$ is equivalent to preservation of the first two moments about any real number.  

For the proof of Proposition \ref{lem_main}, we need the following lemma. 

\begin{lem} \label{nlem}For $g\in C[0,1],$ there exists a unique $f\in C[0,1]$ such that \begin{equation} \label{lemcik} f(x) - 2\int_0^x f(y) \ud y = g(x), \qquad x \in [0,1]. \end{equation} Moreover, \begin{equation}
\label{lemcikw} f(x) = g(x) + 2\int_0^x \e^{2(x-y)} g(y) \ud y. \end{equation} \end{lem}

\begin{proof} For $\lam \in \R,$ the Bielecki-type norm \cite{abielecki,edwards} 
$$ \|f\|_{\lam} = \sup_{x\in [0,1]} |\e^{-\lam x} f(x)|,$$ is equivalent to the original supremum norm. In particular, $C[0,1]$ with $\|\cdot \|_\lam $ is a Banach space. We take $\lam >2$ and consider $T$ mapping $C[0,1]$ into itself, given by 
$$ (Tf)(x) = g(x) + 2\int_0^x f(y) \ud y.$$
Then, for any $f_1,f_2\in C[0,1]$,
\begin{align*}
\|Tf_1 - Tf_2 \|_{\lam } & = \sup_{x\in [0,1]} \left | 2 \int_0^x \e^{-\lam (x-y)} \e^{-\lam y} [f_1(y) - f_2(y)] \ud y\right | \\
& \le \sup_{x\in [0,1]}  2 \int_0^x \e^{-\lam (x-y)} \|f_1 - f_2\|_\lam  \ud y \\
& < \frac 2\lam \|f_1 - f_2\|_\lam . 
\end{align*}
Hence, by the Banach fixed point theorem, there exists a unique $f$ such that $f=Tf$, i.e., \eqref{lemcik} is satisfied. Moreover, a simple calculation shows that $f$ given by \eqref{lemcikw} satisfies \eqref{lemcik}. 
\end{proof}



\begin{prop}\label{lem_main}{For $f \in C[0,1]$, an extension $\tif$ that fulfills  conditions $(a)$-$(c)$, listed above, exists and is uniquely determined.}
\end{prop}

\begin{proof}\bf Step 1. \rm  It suffices to find for all $n\in \mathbb N$ functions $g_n,h_n:[0,1]\to \mathbb R$ related to $\tif $ as follows:
\begin{equation}\label{gnhn} g_n(x) = \tif (x+ n), \quad h_n(x) = \tif (1-x - n), \qquad x \in [0,1].\end{equation} 
Since {in accordance with (a)} we want $\tif $ to be continuous (and in fact well-defined at $x \in \Z$), these functions must satisfy compatibility conditions: 
\begin{equation} \label{cc} h_{n+1} (0)= h_n(1), \quad g_{n+1} (0)= g_n(1), \qquad n \in \mathbb N.\end{equation}
Note that 
\begin{equation}\label{def:h0g0}
g_0(x) = f(x)\quad \hbox{and}\quad h_0 (x) = f(1-x),\qquad x \in [0,1].
\end{equation}

\bf Step 2. \rm 
Condition (b) is satisfied if and only if 
\begin{equation*} M(t): = \int_0^1 \tif (x+ t) \ud x +  \int_0^1 \tif (x- t) \ud x = 
 \int_t^{1+t} \tif (x) \ud x +  \int_{-t}^{1-t} \tif (x) \ud x \end{equation*}
does not depend on $t\in [0,\infty).$ This holds if 
$$ M'(t) = \tif (1+t ) - \tif (t) - \tif (1-t) + \tif (-t) =0, \qquad t \ge 0.$$
Writing $t= n + x $ where $n \in \N, x \in [0,1],$ we check that this  is equivalent to 
\begin{equation} \label{symetria} g_{n+1} + h_{n+1} = g_n + h_n, \qquad n \in \mathbb N.\end{equation}  
Similarly, (c) holds if and only if for all $t\ge 0,$
$$ \int_t^{1+t} (x - t) \tif (x) \ud x + \int_{-t}^{1-t} (x + t) \tif (x) \ud x =2\int_0^{1} x f (x) \ud x.$$
This is satisfied if and only if 
$$ \tif (1+t) - \int_t^{1+t} \tif (x) \ud x - \tif (1-t) + \int_{-t}^{1-t}\tif (x) \ud x = 0, \qquad t \ge 0, $$
i.e., if and only if
$$ g_{n+1}(x) - \int_{n+x}^{1+n + x} \tif (y) \ud y - h_{n} (x) + \int_{-n-x}^{1-n - x} \tif (y) \ud y=0, \ \ \  x \in [0,1], n \in \mathbb N.$$
In other words, (c) holds if and only if 
\begin{align} 
\label{nowyw} g_{n+1} (x) = h_n(x) &+\int_x^1 g_n(y) \ud y + \int_0^x g_{n+1}(y) \ud y \\ & - \int_0^x h_{n+1}(y) \ud y - \int_x^1 h_n(y) \ud y, \qquad  x \in [0,1], n \in \mathbb N.\nonumber
\end{align}

\bf Step 3. \rm Plugging $h_{n+1}= g_n + h_n - g_{n+1}$ into \eqref{nowyw}, we see that (b) and (c) are equivalent to \eqref{symetria} coupled with
\begin{align} 
\label{nowywr} g_{n+1} (x) - 2\int_0^x g_{n+1} (y) \ud y &= h_n(x) -\int_0^1 h_n(y) \ud y - \int_0^x g_{n}(y) \ud y  \\ &\phantom{=}+ \int_x^1 g_n(y) \ud y, \qquad  x \in [0,1], n \in \mathbb N.\nonumber
\end{align}
By Lemma \ref{nlem}, $g_{n+1}$ is uniquely determined by the pair $(g_n,h_n)$, and by \eqref{symetria} so is $h_{n+1}$. Moreover, by \eqref{lemcikw},
\begin{align*} 
g_{n+1} (x) & = r_n (x) + 2\int_0^x \e^{2(x-y)} r_n (y) \ud y \\
&= h_n (x) - \e^{2x} \int_0^1d_n(y) \ud y + 2 \int_0^x \e^{2(x-y)}d_n (y) \ud y,\end{align*}
where $r_n$ is the right-hand side of \eqref{nowywr}, the second equality follows by integration by parts, and
$$ d_n := h_n - g_n.$$  
Combining this with \eqref{symetria}, we obtain the recurrence
\begin{equation}
h_{n+1} = - \psi_{n} + g_n, \qquad g_{n+1} =  \psi_n + h_n, \label{rekurencja}
\end{equation}
where 
\begin{equation}\label{psin}
\psi_n (x):= -\e^{2x}\int_0^1 d_n(y) \ud y + 2 \int_0^x \e^{2(x-y)} d_n(y) \ud y,\qquad x\in [0,1],
\end{equation}
allowing to calculate all $g_n,h_n$'s recursively. 

\bf Step 4. \rm We need to check the compatibility conditions. To this end, we claim that 
\begin{equation} \label{claimik} h_n(1) - g_n(0) = \int_0^1 d_n (y) \ud y =  h_n (0) - g_n (1) , \qquad n \ge 0.\end{equation}
For $n=0$, all the three quantities involved here are zero, since $g_0(0) = f(0) = h_0(1), $ $g_0(1)=f(1)=h_0(0)$, and $\int_0^1 g_0(y)\ud y = \int_0^1 h_0(y)\ud y= \int_0^1 f(y)\ud y.$
Moreover, by \eqref{rekurencja}, introducing $I_n = 2\int_0^1 \e^{2(1-y)}d_n (y) \ud y$, we have
\begin{align*}
\int_0^1 d_{n+1} (y) \ud y &=\e^2 \int_0^1 d_{n} (y) \ud y - I_n, \\
h_{n+1} (1) - g_{n+1} (0)  &= (\e^2 +1) \int_0^1 d_{n} (y) \ud y - I_n - h_n(0) + g_n(1),\\
h_{n+1} (0) - g_{n+1} (1)  & = (\e^2 +1) \int_0^1 d_{n} (y) \ud y - I_n - h_n(1) + g_n(0).\end{align*}
Hence, if \eqref{claimik} holds for some $n$, then it holds for $n+1$, as well, completing the proof of the claim. 

Using \eqref{rekurencja} and \eqref{claimik}, we obtain
$ h_{n+1} (0) = \int_0^1 d_{n} (y) \ud y + g_n (0)= h_n(1),$   and $ g_{n+1} (0) = - \int_0^1 d_{n} (y) \ud y + h_n(0) = g_n (1),$
i.e., both compatibility conditions are satisfied. In other words, $\tif $ can now be \emph{defined} by \eqref{gnhn} and \eqref{def:h0g0}--\eqref{rekurencja},  and is indeed a continuous function. {Because its restriction to $[0,1]$ is clearly $f$, condition (a) is satisfied.} Also, as we have seen, $\tif $ is uniquely determined {by conditions (a), (b), and (c)}.\end{proof}


\begin{defn}\label{def:intext}
Let $f\in C[0,1]$. The function $\tif:\mathbb R\to \mathbb R$ defined in accordance with the rules  \eqref{gnhn}, \eqref{def:h0g0} and \eqref{rekurencja} is called the \emph{integral extension} of $f$. 
\end{defn}

The extension operator 
\[
E: C[0,1] \ni f \mapsto Ef := \tif \in C(\R)
\]
is continuous. Our analysis shows in particular that if \eqref{kelvin} is to define a moments-preserving cosine family, then $Ef$ is necessarily given by \eqref{gnhn}, \eqref{def:h0g0}, and \eqref{rekurencja}. In Theorems \ref{tmain}--\ref{lateron}, later on, we show that \eqref{kelvin} indeed defines a cosine family with the prescribed properties; and that among all cosine families in $C[0,1]$ generated by a realization of the Laplace operator, that defined by \eqref{kelvin} is the only one that preserves the first two moments. 

Let $D$ denote the set of twice continuously differentiable functions $f:[0,1]\to \mathbb R$ satisfying \eqref{bc}
or equivalently 
\begin{equation}F_0(f'')=F_1(f'')=0.\label{bce} \end{equation}

\begin{lem} \label{de} Let $f\in D$. Then its integral extension $Ef$  is twice continuously differentiable on $(-1,2).$ \end{lem}

\begin{proof}  Since by assumption and \eqref{rekurencja} with $n=0$ the integral extension is twice continuously differentiable in the intervals $(-1,0),(0,1)$ and $(1,2)$, we need to check that the left-hand and right-hand derivatives of first and second orders agree at $0$ and $1$; this is the case when
\begin{equation}  g_1'(0) = g_0'(1),\quad  g_1''(0) = g_0''(1), \quad  h_1'(0) = h_0'(1)  \text { and } h_1''(0) = h_0''(1).  \label{pochodne} \end{equation} 
We will prove merely conditions pertaining to $g$'s, the proof related to $h$'s being similar. Using \eqref{rekurencja}, we see that $g_1' (0) = -2 \int_0^1d_0(y)\ud y + 2d_0(0) + h_0'(0)= 2[f(1) - f(0)] - f'(1).$ This equals $g_0'(1)= f'(1)$ by \eqref{bc}. Similarly,  $g''_1 (0) = 4[f(1)-f(0)]-2[f'(0) + f'(1)] + f''(1) = f''(1)=g_0''(1)$.
\end{proof} 

We can finally relate the property of preserving the moments of order 0 and 1 with the boundary conditions~\eqref{bc}.

\begin{thm} \label{tmain} The abstract Kelvin formula $\eqref{kelvin}$ defines a strongly continuous cosine family $(\cic (t))_{t\in \mathbb R}$ on $C[0,1]$. This family preserves both functionals $F_0$ and $F_1$, i.e.,
$$F_i \cic (t) f = F_i f\qquad \hbox{for all } f \in C[0,1]\hbox{ and }t\in \mathbb R, \; i=0,1.$$ 
The generator $A$ of $(\cic(t))_{t\in \mathbb R}$ is given by
\begin{eqnarray}
D(A)&=&D:=\{f\in C^2[0,1]:f'(0)=f'(1)=f(1)-f(0)\},\nonumber \\
Af&=&f''. \label{opa}
\end{eqnarray}
\end{thm}

\begin{proof} Let $R: C(\R) \to C[0,1]$ map a member of $C(\R)$ to its restriction to $[0,1]$. Then \eqref{kelvin} takes the form 
\begin{equation}\label{cp} \cic (t)= RC(t)E,  \qquad  t \in \R. \end{equation}

By \eqref{rekurencja}, a pair $(g_{n+1},h_{n+1})$ is obtained  from $(g_{n},h_{n})$ by means of a bounded linear operator mapping $C[0,1]\times C[0,1]$ into itself. Since for any $t,$ $RC(t)Ef $ depends merely on the finite number of such pairs, it follows that $\cic (t)$ is a bounded linear operator in $C[0,1]$. That the operators $\cic (t)$ preserve functionals $F_0$ and $F_1$ is clear by Proposition \ref{lem_main}.

Fix $f \in C[0,1]$ and $s \in \R.$ Clearly, $C(s)Ef$ extends $RC(s)Ef$ and, by the cosine equation for $C$ and the definition of $Ef$, we have 
\[
F_i C(t) C(s) Ef= F_i f= F_i RC(s)Ef,\qquad i=0,1,\; t \in \R.
\]
By uniqueness of integral extensions, this shows that $C(s)Ef$ is the integral extension of $RC(s)Ef$:
 \[
 ERC(s)Ef = C(s)Ef,\qquad s\in \mathbb R.
 \]
  Using this and the cosine equation for $C$, we check that 
  \[
  2\cic (t)\cic (s) f = \cic (t+s)f + \cic (t-s) f,\qquad  t,s \in \R,
  \] i.e., that $\cic $ is a cosine family. This family is strongly continuous, i.e., we have $
\lim_{t\to 0} R C(t) Ef = f$ for all $f\in C[0,1]$, since $Ef$, as restricted to any compact interval,  is a uniformly continuous function, and on $[0,1]$ it coincides with $f.$ 

Turning to the characterization of the generator: 
Lemma \ref{de} and the Taylor formula imply that for $f \in D,$
$$\lim_{t\to 0} \frac 2{t^2} (C(t) \tif  (x) - \tif (x)) = \tif''(x),  \qquad x \in (-1,2);$$ the limit is uniform in $x\in [0,1]$ since $\tif''$ is uniformly continuous in any compact subinterval of $(-1,2)$. By \eqref{cp} this proves that $f$ belongs to $D(A)$ and we have $Af = f''.$   

Finally, we check that there is a $\lam >0$ such that for all $g \in C[0,1]$ there exists $f \in D$ such that $\lam f - f'' = g.$ Since $\lam - A$ is injective for some large $\lam $ and its range is $C[0,1]$, this will show that $D$ cannot be a proper subset of $D(A)$ (see e.g. \cite{kniga} p. 267). 

The general solution to this ordinary differential equation is 
\[
f(x) = C_1 \e^{\sqrt \lam x} + C_2 \e^{-\sqrt \lam x} - \frac 1{\sqrt{\lam}}\int_0^x \sinh [\sqrt \lam (x- y) ] g(y) \ud y.
\]
Such an $f$ satisfies \eqref{bc} if and only if $C_1$ and $C_2$ satisfy the following system of equations:
\begin{align*} 
(\slam - \e^{\slam} + 1)C_1 &+ (1- \slam  - \e^{-\slam}) C_2 \\ &= -\frac 1{\slam } \int_0^1 \sinh (1-y) g(y) \ud y,\\
(1+\slam \e^{\slam}  - \e^{\slam})C_1 &+ (1- \slam\e^{-\slam} - \e^{-\slam}) C_2 \\ & = 
 \int_0^1 \cosh (1-y) g(y) \ud y 
 -\frac 1{\slam } \int_0^1 \sinh (1-y) g(y) \ud y. \end{align*}
The (unique) choice of such $C_1$ and $C_2$ is possible, since the determinant of this linear system equals $4 \slam + 2\lam \cosh \slam - 4\slam \sinh \slam \not \equiv 0$.   \end{proof}

Recapitulating the results of this section, we see in particular that requiring that a cosine family generated by a one-dimensional Laplace operator preserves the functionals \eqref{fi} is (at least in the context of Kelvin's method of images) equivalent to assuming the  boundary conditions \eqref{bc}. In other words, there is a unique cosine family of the form \eqref{kelvin} (with continuous extension operator $E$) that preserves both functionals $F_0$ and $F_1$: this is the cosine family generated by $A$, the one-dimensional Laplace operator with boundary conditions \eqref{bc}. 

However, \eqref{bce} suggests a stronger result. Before presenting it, we recall (see, e.g., \cite[proof of
Theorem  3.14.17]{abhn} or \cite[Theorem  8.7]{goldstein}), that if $({\rm Cos}(t))_{t\in \mathbb R} $ 
is a strongly continuous cosine family in a Banach space $X$, then the  abstract
Weierstrass formula
\begin{equation}
  \label{eq:1}
  S(t)x = \frac{1}{\sqrt{\pi t}} \int_0^{\infty} \e^{-\tau^2/4t}
  {\rm Cos}(\tau)x \ud \tau
  \quad
  (t > 0, \ x \in X),
\end{equation}
coupled with $S(0)=Id_X$ (identity operator in $X$), defines a strongly continuous semigroup $(S(t))_{ t\ge 0}$ with the generator equal to the generator of the cosine family.

\newcommand{\ap}{A_{\rm{p}}}
\begin{thm}\label{lateron} Among all semigroups generated by one-dimensional Laplace operators in $C[0,1]$, there is only one that preserves the moments of order 0 and 1, and this is the semigroup with domain given by boundary conditions \eqref{bce}, i.e. the one given by the Weierstrass formula applied to the cosine family \eqref{kelvin}.  \end{thm}

\begin{proof} Let $\sem{\ap}$ be a moments-preserving semigroup generated by a one-di\-men\-sional Laplace operator in $C[0,1],$  and let $f \in D(\ap)$. It suffices to show that $f$ satisfies the boundary conditions \eqref{bce}. Consider 
\[
u_i (t) := F_i (\e^{t\ap}f ),\qquad i =0,1,\; t\ge 0.
\]
The scalar-valued functions $u_i$ are differentiable with $u_i'(t)= F_i (\ap \e^{t\ap}f)=0,$ the last equality following from the fact that the semigroup preserves the moments. Taking $t=0$ and noting that $\ap f= f''$, we obtain \eqref{bce}. \end{proof}

It goes without saying that this proposition implies uniqueness of the moments-preserving cosine family as well: by Weierstrass formula, if a cosine family preserves the moments, then so does the corresponding semigroup. Since there is only one moments-preserving semigroup, there can be no more cosine families. 

\begin{rem}
The argument used in the proof of Theorem~\ref{lateron} shows that, if a semigroup or a cosine family generated by a realization of the Laplace operator preserves the moment of order $i$, then for $f$ in the domain of the generator we have \begin{equation}F_i(f'')=0. \label{wry} \end{equation} 
Using the identity 
\[
F_i (f') = f(1) - iF_{i-1} f,
\]
which holds for all $ i \ge 1$ and all $f \in C^1[0,1]$, we check that~\eqref{wry} holds if and only if \begin{equation}
\begin{array}{rcll} 
f'(0)&=&f'(1) \qquad&\hbox{if }i =0,\\
f'(1) &=& f(1) - f(0) \qquad&\hbox{if }i =1, \\ 
 if(1)&= & i(i-1)F_{i-2}f - f'(1). \qquad&\hbox{if }i \ge 2.
\end{array} \label{unki} 
\end{equation}
A big question is of course if the requirement that two moments, say $F_i$ and $F_j$ ($i\not =j$), are preserved,  determines a cosine family or a semigroup generated by a Laplace operator.
A generation theorem for semigroups related to moments $F_0$ and $F_j$ with arbitrary $j\not =0$ has been obtained in~\cite[Thm 3.4]{MugNic12}.
\end{rem}

\section{Decomposition of $\cic$}\label{s2}

\begin{center}
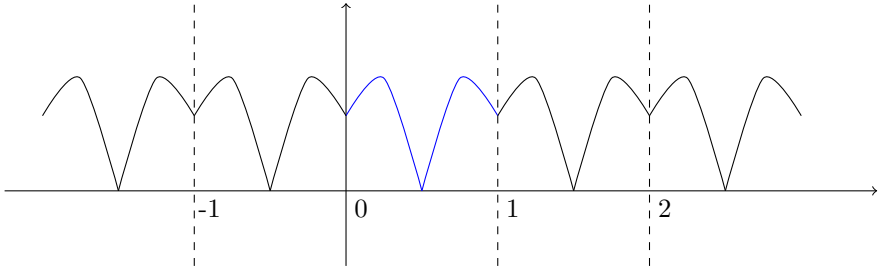
\begin{figure}
\begin{tikzpicture}
\draw [->] (-4.5,0)--(7,0); 
\draw [->] (0,-1)--(0,2.5);
\draw [dashed] (2,-1)--(2,2.5);
\draw [dashed] (4,-1)--(4,2.5);
\draw [dashed] (-2,-1)--(-2,2.5);
\draw [color=blue] plot [smooth] coordinates {(0,1) (0.5,1.5) (1,0)};
\draw [color=blue] plot [smooth] coordinates {(1,0) (1.5,1.5) (2,1)};
\draw plot [smooth] coordinates {(0,1) (-0.5,1.5) (-1,0)};
\draw plot [smooth] coordinates {(-1,0) (-1.5,1.5) (-2,1)};
\draw plot [smooth] coordinates {(4,1) (4.5,1.5) (5,0)};
\draw plot [smooth] coordinates {(5,0) (5.5,1.5) (6,1)};
\draw plot [smooth] coordinates {(-4,1) (-3.5,1.5) (-3,0)};
\draw plot [smooth] coordinates {(-3,0) (-2.5,1.5) (-2,1)};
\draw plot [smooth] coordinates {(4,1) (3.5,1.5) (3,0)};
\draw plot [smooth] coordinates {(3,0) (2.5,1.5) (2,1)};
\node [below] at (2.2,0) {1};
\node [below] at (0.2,0) {0};
\node [below] at (4.2,0) {2};
\node [below] at (-1.8,0) {-1};
\end{tikzpicture}\caption{The integral extension of an $f\in \cev$.}\label{fig1}
\end{figure}
\end{center}

Let $\cod , \cev \subset C[0,1]$ be the (closed) subspaces of functions $f$ with graphs that are, respectively, \textit{asymmetric} and \textit{symmetric} about $\frac 12,$ i.e.,
\[
f(1-x) = -f(x)\qquad \hbox{and}\qquad f(1-x) = f(x),\qquad x \in [0,1],
\]
respectively. Recall that each $f \in C[0,1]$ is a sum $f= \fod + \fev $ of its odd and even parts (defined by 
\[
\fod (x) := \frac 12[f(x)- f(1-x)]
\]
and 
\[
\fev (x) := \frac 12 [f(x) + f(1-x)],
\]
respectively), and this representation is unique.  Moreover, the maps $f\mapsto \fev $ and $f\mapsto \fod $ are projections of norm one onto $\cev$ and $\cod $, respectively. We can introduce in the same way the spaces $C_{\rm odd}(\R)$ and $C_{\rm even}(\R)$ of all continuous functions on $\R$ that are, respectively, asymmetric and symmetric about $\frac 12$. (Observe that $C_{\rm odd}(\R)$ and $C_{\rm even}(\R)$ are in general not asymmetric and symmetric about $0$, hence they are not odd or even, respectively, in the usual sense.)

Symmetries hidden in \eqref{rekurencja} allow reducing analysis to the subspaces $\cod $ and $\cev$, treated separately.  We begin with the following observation, where $(f_n)_{n\in \mathbb N},(g_n)_{n\in \mathbb N}$ are the function families introduced in the proof of Proposition~\ref{lem_main}.

\begin{lem} \label{lemacik} Suppose $f \in \cev$. Then
\begin{equation} \label{parzyste} h_n=g_n= f, \qquad n \in \mathbb N. \end{equation}
If $f\in \cod$, then 
\begin{equation} \label{parzyste2} h_n+g_n=0. \qquad  n \in \mathbb N. \end{equation}
Accordingly, $\tif\in C_{\rm even}(\R)$ if $f\in \cev$, whereas $\tif\in C_{\rm odd}(\R)$ if $f\in \cod$.
\end{lem}

\begin{proof} Let $f\in \cev$. We proceed by induction. For $n=0$ the claim is true by definition. Suppose~\eqref{parzyste} holds for some $n\in \mathbb N.$ Then $d_n=0$, and $\psi_n$ introduced in \eqref{rekurencja} vanishes. Hence, the latter formula implies $h_{n+1} = g_n = f$ and $g_{n+1} = h_n = f$.
Therefore, $\tif$ is even: $\tif (x + n)= g_n(x) = h_n(x)= \tif (1-x-n), n\in \N, x \in [0,1].$

Now, let $f\in \cod$. Since by assumption $g_0 + h_0 =0$, \eqref{parzyste2} can be deduced from \eqref{symetria} by induction. Hence,
$\tif(x+n)=g_n(x)=-h_n(x)=-\tif(1-(x+n)), n\in \mathbb N, x\in [0,1].$
This completes the proof. \end{proof}
\begin{ex}\label{example.1}
A typical graph of the integral extension of an $f \in \cev$ is depicted in Figure \ref{fig1}. In particular for even functions, $\sup_{x \in \R}|\tif (x) | = \sup_{x \in [0,1]}|f(x)|.$ The integral extensions of odd functions do not have the latter property; see Figure \ref{fig2}. In fact, an integral extension of an odd function will typically be unbounded: for example, a direct computation shows that  $\tif(x):= 2x - 1$, $x \in \mathbb R$, is the integral extension of $f(x):= 2x - 1$, $x\in [0,1]$.
\end{ex}

\begin{center}
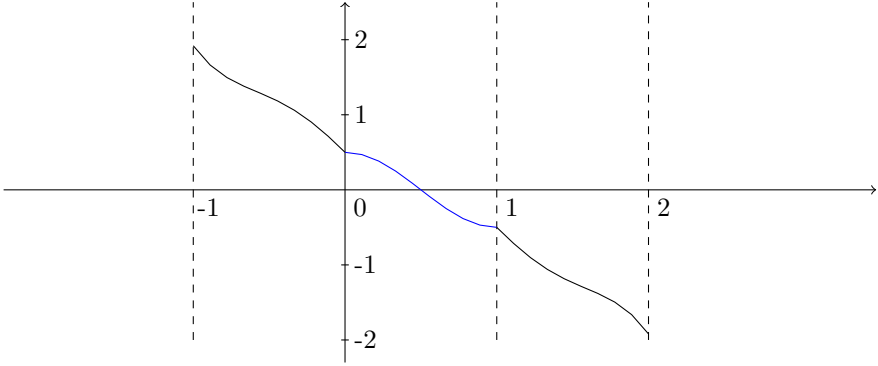
\begin{figure}
\begin{tikzpicture}
\draw [->] (-4.5,0)--(7,0); 
\draw [->] (0,-2.3)--(0,2.5);
\draw [dashed] (2,-2)--(2,2.5);
\draw [dashed] (4,-2)--(4,2.5);
\draw [dashed] (-2,-2)--(-2,2.5);
\node [below] at (2.2,0) {1};
\node [below] at (0.2,0) {0};
\node [below] at (4.2,0) {2};
\node [below] at (-1.8,0) {-1};
\draw (-.05,1)--(.05,1);
\draw (-.05,2)--(.05,2);
\draw (-.05,-1)--(.05,-1);
\draw (-.05,-2)--(.05,-2);
\node [right] at (0,1) {1};
\node [right] at (0,2) {2};
\node [right] at (0,-1) {-1};
\node [right] at (0,-2) {-2};
\draw[blue, domain=0:1, samples=10] plot (2*\x, {0.5*cos(3.141*\x r)});
\draw[black, domain=1:2, samples=10] plot (2*\x, {0.1442*(2 *cos (3.141*(\x-1) r) -  3.141 * sin (3.141 * (\x-1) r)- 2*exp(2*(\x-1) ) ) - 0.5*cos(3.141 *(\x-1) r)});
\draw[black, domain=-1:0, samples=10] plot (2*\x,  {-0.1442*(2 *cos (-3.141*\x r) -  3.141 * sin (-3.141 *\x r)- 2*exp(-2*\x ) ) + 0.5*cos(-3.141 *\x r)});
\end{tikzpicture}\caption{The integral extension of the $f\in \cod$  given by $f(x):=\frac 12 \cos (\pi x)$; in this case $g_1(x)= \frac 2{\pi^2 + 4} (2\cos (\pi x) - \pi \sin (\pi x) - 2\e^{2x}) - \frac 12 \cos (\pi x).$}\label{fig2}
\end{figure}
\end{center}
\vspace{-0.7cm}
\begin{prop}\label{p2} 
\begin{enumerate}[{\rm (a)}]
\item The cosine family $(\cic(t))_{t\in \mathbb R}$ leaves the subspace \linebreak $\cev$ invariant, and is a cosine family of contractions there. The generator $A_0$ of $(\cic(t))_{t\in \mathbb R}$ restricted to  $\cev$ is given by $A_0f = f''$ with domain composed of twice continuously differentiable even functions satisfying $f'(0)=0$.
\item  The cosine family $(\cic(t))_{t\in \mathbb R} $ leaves the subspace $\cod$ invariant. The generator $A_1$ of $(\cic(t))_{t\in \mathbb R}$ restricted to  $\cod$ is given by $A_1f = f''$ with domain composed of twice continuously differentiable odd functions satisfying $f'(0)=-2f(0).$ 
\item Denoting by $(C_{A_0}(t))_{t\in \mathbb R}$ and $(C_{A_1}(t))_{t\in \mathbb R}$ the cosine family $(\cic(t))_{t\in \mathbb R}$ as restricted to the subspaces $\cev $ and $\cod $, respectively, we have 
\begin{equation} \label{rozklad} \cic (t) f = C_{A_0}(t) \fev + C_{A_1}(t) \fod , \qquad f \in C[0,1], t \in \R. \end{equation}
\end{enumerate} \end{prop}

\begin{proof} (a) Let $f\in \cev$. For $t \in \R$ and $x \in [0,1]$, 
\begin{eqnarray} \label{rachun} 
\cic (t)f (1-x)& =& \frac 12 [\tif (1-x+ t) + \tif(1- x- t)]\nonumber \\
& =& \frac 12 [\tif (x-t) + \tif (x+t)] \nonumber \\
& =& \cic (t) {f} (x) ,\end{eqnarray}
due to the fact that $\tif$ is even.	
By \eqref{parzyste}, $\sup_{x\in \R} |\tif (x) | = \sup_{x\in [0,1]}|f(x)|.$ Hence, 
$$\sup_{t\in \R}  \| \cic (t) f \| \le \sup_{t \in \R} \sup_{x\in [0,1]}   \frac 12 \big|\tif (x-t) + \tif (x+t)\big|  \le \sup_{x\in \R} |\tif (x) | = \|f\|.$$  The claim concerning $A_0$  can be deduced from the fact that $A_0$ is the part of $A$ in  $\cev $ or directly from \eqref{parzyste} -- the integral extension of a twice continuously differentiable $f\in \cev$ is twice continuously differentiable if $f'(0)=0.$

(b) By Lemma \ref{lemacik}, $\tif $ is odd.  Therefore, the first claim follows by a calculation similar to \eqref{rachun}. The characterization of the generator is also proved as in (a). (c) is an immediate consequence of (a) and (b). \end{proof}

It is perhaps worth noting that \eqref{parzyste} implies that for $f\in \cev$, $\tif $ is periodic with period 1: $\tif (x+1) = f(x), x \in \R.$ Hence, $\cic (t+1) f = \cic (t)f$. In contrast, behavior of $(\cic(t))_{t\in \mathbb R}$ as restricted to $\cod$ is not so evident: in particular, it is even unclear if the cosine family is bounded on this subspace.

Let us take a closer look at the even part of the moments-preserving cosine family. The generator $A_0$ of this part, described in Proposition \ref{p2} (a), has a natural extension to a densely defined operator in $C[0,1]$, which by an abuse of notation we denote $A_0$ again, given by
\begin{eqnarray*}
D(A_0)&:=&\{f\in C^2[0,1]:f'(0)=f'(1)=0\},\\
A_0 f&:=&f''.
\end{eqnarray*}
This is of course a very-well known object, the one-dimensional Neumann Laplacian: the generator 
of a cosine family in $C[0,1]$, which we denote again $(C_{A_0}(t))_{t\in \mathbb R}$,  and of the related Feller  semigroup  of the Brownian motion with two reflecting barriers at $0$ and $1$. The cosine family generated by $A_0$ is given by the abstract Kelvin formula \eqref{kelvin} with $\tif $ denoting the unique extension of $f$ whose graph is symmetric about $0$ and $1$ (see \cite[pp. 21-13]{weinberger}, cf.\ also~\cite[pp. 340-342]{feller}, where the case of the related semigroup is covered). Since for $f \in \cev $, this extension coincides with the integral extension (see Lemma  \ref{lemacik}, and Figure \ref{fig1}), the cosine families $(C_{A_0}(t))_{t\in \mathbb R}$ and $(\cic (t))_{t\in \mathbb R}$
coincide on $\cev .$ In particular, $C_{A_0}(t)$ leaves $\cev$ invariant for all $t\in \mathbb R$, and  we obtain the following, somewhat unexpected corollary. 

\begin{cor} \label{wnosek1} For $f \in \cev $, the first two moments are preserved along the trajectory $ t\mapsto C_{A_0}(t) f $ of the cosine family $(C_{A_0}(t))_{t\in \mathbb R}$ generated by the Neumann Laplacian. \end{cor}  

Similarly, $A_1$, the generator of the odd part, can be naturally extended to the following densely defined operator in $C[0,1]$:
\begin{eqnarray*}
D(A_1)&:=&\{f\in C^2[0,1]:f'(0)=-2f(0), f'(1)=2f(1)\},\\
A_1 f&:=&f''.
\end{eqnarray*}
This operator generates a cosine family in $C[0,1]$ (see the main theorem in \cite{chinczyki} or in \cite{ChiKeyWar07}), denoted again $(C_{A_1}(t))_{t\in \mathbb R}$. The latter cosine family is given by the abstract Kelvin formula, where the integral extension is replaced by $\tif $ given for $x \in [-1,2]$ by (see \cite[Lemma 3.1]{ChiKeyWar07})  
$$ \tif (x)   = \begin{cases} f(-x) + 4 \e^{-2x} \int_0^{-x} \e^{-2y} f(y) \ud y , & x \in [-1,0), \\
f(x), & x \in [0,1], \\
f(2-x) + 4 \e^{2(x-1)} \int_0^{x-1} \e^{-2y} f(1-y)\ud y , & x \in (1,2].
\end{cases} $$
It follows that the graph of $\tif $ (as restricted to the interval $[-1,2]$) is asymmetric about $\frac 12$, provided $f \in \cod .$ Hence the operators $C_{A_1}(t), t \in [-1,1]$, leave $\cod $ invariant and the same is true for all $t \in \R $ by the cosine equation. This together with Proposition \ref{p2} (b), shows that $(C_{A_1}(t))_{t\in \mathbb R}$ and $(\cic(t))_{t\in \mathbb R} $ coincide on $\cod .$  

The analysis presented above provides a slightly different meaning for
\eqref{rozklad}. While in Proposition \ref{p2} (c), $(C_{A_0}(t))_{t\in \mathbb R}$ and $(C_{A_1}(t))_{t\in \mathbb R}$ were interpreted as the cosine families being restrictions of $(\cic(t))_{t\in \mathbb R} $ to the subspaces of even and odd functions, respectively, now we have proved that they can also be seen as the cosine families defined on the whole of $C[0,1].$ Moreover, we obtain the following analogue of Corollary \ref{wnosek1}. 

\begin{cor} \label{wnosek2} For $f \in \cod $, the first two moments are preserved along the trajectory $t \mapsto C_{A_1}(t)f$ of the cosine family $(C_{A_1}(t))_{t\in \mathbb R}$. \end{cor} 

In this context, formula \eqref{rozklad} may be expressed by saying that the cosine family $\cic $ is a direct sum of two moments-preserving cosine families acting in the subspaces of even and odd functions, respectively. Our final result in this section, presented below, says that the summands in \eqref{rozklad} are uniquely determined (when restricted to the related subspaces); we state the result in the more general context of  semigroups, since it implies the corresponding statement for cosine families by the Weierstrass formula.  
\newcommand{\aev}{A_{\text{\rm even}}}
\newcommand{\aod}{A_{\text{\rm odd}}}
\begin{prop}\label{unique} \  
\begin{enumerate}[{\rm (a)}]
\item Let $\aev$ be the part of the Laplacian in $\cev$: i.e. let it be the restriction of the Laplacian to the domain composed of all even, twice continuously differentiable  functions. From among all semigroups generated by restrictions of $\aev$ in $\cev$ there is only one that preserves $F_0$, the moment of order zero. This is the semigroup generated by $A_0$, defined in Proposition \ref{p2}. This semigroup preserves the moment of first order as well. 
\item Let $\aod$ be the part of the Laplacian in $\cod$: i.e. let it be the restriction of the Laplacian to the domain composed of all odd, twice continuously differentiable  functions. From among all semigroups generated by restrictions of $\aod$ in $\cod $ there is only one that preserves $F_1$, the moment of order one. This is the semigroup generated by $A_1$, defined in Proposition \ref{p2}. Clearly, this semigroup preserves the moment of order zero as well. 
\end{enumerate}
 \end{prop}

\begin{proof} (a) Let $f$ be a member of the domain of the generator, say $A_p$, of a semigroup that preserves the moment of order zero. Since $f$ is even, $f'(0)=-f'(1),$ and since the semigroup preserves the moment of order zero, $f$ satisfies the first equation in \eqref{unki}. It follows that $f'(0)=f'(1)=0$, i.e. that $f\in D(A_0).$ Since $A_0$ cannot be a proper extension of another generator, we see that $A_0=A_p$. 

(b) The proof is analogous to (a) and we omit it; note that, in fact, any cosine family in $\cod $ preserves the moment of order zero, for $F_0f =0$ for all $f \in \cod.$ \end{proof}

The statement on automatic preservation of the first moment in point (a) may seem surprising at a first glance. It becomes clear, however, once we note that 
\begin{equation} \label{clear} 2F_1 f = F_0 f, \qquad f \in \cev , \end{equation}
the latter relation being a direct consequence of 
\[
F_1 f = \int_0^1 x f(x) \ud x = \int_0^1 (1-x) f(x) \ud x\qquad \hbox{ for } f \in \cev.
\] 

\section{Asymptotic behavior of the related semigroup}\label{s3}

Let $\sem{A}$ be the moments-preserving semigroup generated by the operator $A$ defined in \eqref{opa}; in this section we provide information on asymptotic behavior of this semigroup.  

Let  $\sem{A_0}$ and $\sem{A_1}$ denote the semigroups generated by the operators $A_0$ and $A_1$, respectively, described in Proposition \ref{p2}.  The Weierstrass formula shows that  the spaces $\cod $ and $\cev$ are invariant under $\sem{A}$, and by \eqref{rozklad}, we have
\begin{equation} \label{rozklads} \e^{At} f = \e^{A_0t}\fev + \e^{A_1t}\fod , \qquad f \in C[0,1],\; t \ge 0. \end{equation}
It is well-known that (see e.g., the general homogenization theorem of Conway, Hoff and Smoller \cite{chs} or \cite[Thm 14.17]{smoller}; the particular case considered here may also be deduced 
from the explicit semigroup expression in \cite[p. 68 eq. (2.8)]{EngNag00})   
\[
\lim_{t\to \infty} \e^{tA_0} f = (F_0 f)f_0 
,\qquad f\in C[0,1],
\]
where 
\begin{equation}
\label{def:f0}
f_0:=1_{[0,1]}.
\end{equation}
 In fact, there is a positive constant $\eps >0$ such that 
 \[
 \| \e^{tA_0} - P_0 \|_{\mathcal L(C[0,1])} \le \e^{-\eps t},
 \]
  where 
  \[
  P_0f: =  \left(\int_0^1 f(x) \, \mathrm {d} x\right)f_0.
  \]
By \eqref{rozklads}, this implies existence of the limit of the even part of the semigroup $\sem{A}$: 
 \begin{equation}\label{even} \lim_{t\to \infty} \e^{tA_0} \fev = \left(\int_0^1 \fev(x) \, \mathrm {d} x\right )f_0 = \left(\int_0^1 f(x) \, \mathrm {d} x\right )f_0.\end{equation}

Hence, it remains to determine the limit of the odd part {in order to determine the long-time behavior of the whole system.} 
We begin by noting that $\cod $ is isometrically isomorphic to $\cez $, the space of continuous functions on $[0,1]$ vanishing at $x=0.$ The isomorphism  is given by 
 \begin{eqnarray*}
I: \cod &\to &\cez,\\
 If(x) &:=& f\left(\frac {1-x}2\right),
  \end{eqnarray*}
with inverse given by
\[
I^{-1}f(x) =
\left\{
\begin{array}{ll}
f(1-2x),\qquad &x\in [0,\frac 12),\\
-f(2x-1),&x \in [\frac 12,1].
\end{array} 
\right.
\]
The isomorphic image of $A_1$ in $\cez $ is given by $B_1 f= I A_1 I^{-1}f$ with domain equal to the image of the domain of $D(A_1),$ i.e. 
\begin{eqnarray*}
D(B_1)&=&\{f\in C^2[0,1]:f(0)=f''(0)=0\hbox{ and }f'(1)=f(1)\},\\
B_1f&=&4f''.
\end{eqnarray*}

The {strongly-continuous} semigroup generated by $B_1$, and the related cosine family, were already considered in \cite{kelvin}, even in the context of Lord Kelvin's method of images.  This semigroup describes chaotic movement of particles in the interval $[0,1]$ with constant inflow of particles from the boundary at $x=1$ and outflow at the boundary $x=0$ (see \cite{kelvin,kazlip}). As it turns out, the rates of inflow and outflow are so tuned here that in the limit a non-trivial equilibrium is attained. Now, the null space of $B_1$ is the linear span of $h(x)=x$ and this suggests that the odd part converges to a scalar multiple of $I^{-1}h (x)= 1-2x$. 
Since the odd part preserves $F_1$, 
and $F_1 I^{-1} h= - \frac 16, $ it will be convenient to work with $f_1$ defined by 
\begin{equation}
\label{def:f1}
f_1 (x) = 12x - 6,\qquad x\in [0,1],
\end{equation}
{which is normalized so that $F_1 f_1=1$.}
To recapitulate, our aim is to show 
\begin{equation} \label{glim} \lim_{t\to \infty}  \e^{tA} f = Pf \end{equation} where 
\begin{equation} \label{pe} Pf = (F_0\fev )f_0 + (F_1\fod )f_1,\end{equation}  
i.e., that in the limit the moments-preserving semigroup forgets the shape of the initial value and remembers merely its first two moments about $0$: note that
$$ F_0 Pf = F_0 \fev = F_0 f \quad \text{ and } \quad F_1 Pf = \frac 12 F_0\fev + F_1 \fod = F_1 f, $$
where the last equality follows by \eqref{clear}.

In the main result of this section, Theorem~\ref{equilibr-nonneum} (later on),
we are actually going to prove a stronger result than (\ref{glim}), and our proof will be based on methods different from those that have been used so far. We will namely lift our $A$ to a suitable  operator in $L^2(0,1)$, prove some spectral results there by Hilbert space methods, and finally return to the original operator $A$, showing norm convergence towards a sum of two projections for the semigroup generated by it. 

\begin{lem}\label{lem_generA}
Consider the operator $\tilde{A}$ defined by
	\begin{eqnarray*}
D(\tilde{A})&:=&\left\{f\in H^2(0,1):
f'(0)=f'(1)=f(1)-f(0)\right\},\\
\tilde{A} f&:=&f''.
\end{eqnarray*}
Then: 
\begin{enumerate}[(1)]
\item $\tilde{A}$ generates a compact, analytic, self-adjoint semigroup on $L^2(0,1)$.
\item {The spectrum of $\tilde{A}$ consists of countably many negative eigenvalues accumulating at $-\infty$.}  The largest eigenvalue of $\tilde{A}$ is $0$, which has multiplicity 2.
\end{enumerate}
\end{lem}

\begin{proof}
{Observe that our boundary conditions can be written as
\begin{equation}
\label{bc-behlug}
{\mathcal A}\underline{f}+
{\mathcal B}\underline{f'}=0,	
\end{equation}
where
\[
{\mathcal A}:=\begin{pmatrix}
-1 & 1\\ 1 & -1
\end{pmatrix},\qquad 
{\mathcal B}:=\begin{pmatrix}
1 & 0 \\ 0 & 1
\end{pmatrix}
\]
and
\[
\underline{f}:=\begin{pmatrix}
f(0)\\ f(1)
\end{pmatrix}
\qquad \hbox{ and }\qquad \underline{f}':=\begin{pmatrix}
-f'(0)\\ f'(1)
\end{pmatrix}.
\]

(1) A direct computation (e.g.\ on the lines of~\cite[Thm.~3.4.8]{Are00}) shows that $\tilde{A}$ is the operator associated with the densely defined sesquilinear form
$$a(f,g):=
(f'|g')_{L^2} +({\mathcal A}\underline{f}|\underline{g})_{\mathbb C^2},\qquad f,g\in V:=H^1(0,1).$$
(Recall that by definition the operator associated with the densely defined form $a$ is given by
\begin{eqnarray*}
D(T)&:=&\{f\in V:\exists g\in L^2(0,1) \hbox{ s.t. }a(f,h)=(g,h)\hbox{ for all }h\in V\},\\
Tf&:=&-g;
\end{eqnarray*}
hence what we claim is that $\tilde A=T$.)}

{ If $f \in V$, then by the interpolation inequality of Gagliardo--Nirenberg (\cite[Comment 8.1.(iii)]{Bre10}) there holds for some constant $C>0$ and   all $\epsilon \in (0,1)$
\begin{equation}
\label{imbed}
\|f\|^2_{C[0,1]} \le C \|f'\|_{L^2} \|f\|_{L^2}\le \frac{C\epsilon}{2} \|f'\|_{L^2}^2 +\frac{C}{2\epsilon}\|f\|_{L^2}^2,\qquad f\in V.
\end{equation}
thanks to the inequality of Schwarz and Young. }
Now, the hermitian matrix $\mathcal A$ has eigenvalues $0$ and $-2$, hence 
\[
({\mathcal A}\underline{f}|\underline{f})_{\mathbb C^2}\ge -2 \left(\|f(0)\|^2+\|f(1)\|^2\right)\ge -4\|u\|^2_{C[0,1]}.
\]
This implies that for all $f\in V$
{
\begin{align*} a(f,f)  & \ge \|f'\|_{L^2}^2  -4\|f\|^2_{C[0,1]}\\ & \ge (1-2C\epsilon) \|f'\|_{L^2}^2- \frac{2C}{\epsilon} \| f\|_{L^2}^2.\end{align*}
}
Taking $\epsilon\in (0,\frac{1}{2C})$ we finally conclude that $a$ is elliptic in the sense of~\cite[\S~7.2]{Are06}.

Because $a$ is clearly bounded and symmetric, too, we conclude by~\cite[Thm.~7.1.5 and \S~7.2]{Are06} that $A$ generates an analytic, self-adjoint, quasi-con\-trac\-tive semigroup on $L^2(0,1)$. Since by the Rellich--Khondrachov Theorem (see e.g. \cite[Thm.~8.8]{Bre10}), $V$ embeds compactly in $L^2(0,1)$, this semigroup is compact by~\cite[Prop.~8.1.8]{Are06}. 

(2) A direct computation shows that $0$ is an eigenvalue of $\tilde A$ and the null space of $\tilde A$ is a 2-dimen\-sional space (which has $(f_0,f_1)$ as an orthogonal basis, where $f_0,f_1$ are defined in~\eqref{def:f0} and~\eqref{def:f1}).
Using the result recently obtained in~\cite{BehLug10} we will show that $\tilde A$ has no strictly positive eigenvalues. A special case of~\cite[Thm.~1]{BehLug10} states that the number of strictly positive eigenvalues of the second derivative with boundary conditions of the form \eqref{bc-behlug} agrees with the number of strictly positive eigenvalues of the matrix ${\mathcal A}{\mathcal B}^* +{\mathcal B}{\mathcal M}_0 {\mathcal B}^*$, where $\mathcal A$ and $\mathcal B$ are the matrices that appear in \eqref{bc-behlug} and ${\mathcal M}_0$ is defined in accordance with~\cite[eq.~(5)]{BehLug10}. Since in our case the latter matrix happens to be equal to $\mathcal A,$ 
$
{\mathcal A}{\mathcal B^*} +{\mathcal B}{\mathcal M_0}{\mathcal  B^*} =2 {\mathcal A}$
and this matrix has eigenvalues $0$ and $-2$.  It follows that $\tilde{A}$ has no strictly positive eigenvalues. In particular, $0$ is the spectral bound of $\tilde{A}$. The remaining assertion follows from the general spectral theory of self-adjoint operators with compact resolvent.
\end{proof}

\begin{lem}\label{lemmengelA}
The operator $A$ on $C[0,1]$ is the generator of a compact, analytic semigroup  $(\e^{tA})_{t\ge 0}$ that is the restriction of $(\e^{t\tilde{A}})_{t\ge 0}$ to $C[0,1]$. The spectral bound of $A$ is 0.
\end{lem}

\begin{proof}
 A direct computation shows that the operator $A$ is the part of $\tilde{A}$ in $C[0,1]$. Because the semigroup $(\e^{t\tilde{A}})_{t\ge 0}$ is analytic, 
\begin{equation}\label{invar1}
\e^{tA}C[0,1]\hookrightarrow \e^{t\tilde{A}}L^2(0,1)\subset D(\tilde{A}^{\; 2}),\qquad t>0,
\end{equation}
and hence the semigroup leaves $C[0,1]$ invariant,  as 
\begin{equation}\label{invar2}
D(\tilde{A}^{\; 2})\hookrightarrow D(A)\hookrightarrow C[0,1].
\end{equation}
Finally, the part of $\tilde{A}$ in $V$ generates a strongly continuous semigroup, hence strong continuity of the semigroup generated by $A$ on $C[0,1]$ can be proved exploiting~\eqref{imbed}.
 Then, it follows from~\cite[\S~II.2.2]{EngNag00} that the semigroup obtained by restricting $(\e^{t\tilde{A}})_{t\ge 0}$ to $C[0,1]$ is generated by the part of $\tilde{A}$ in $C[0,1]$, viz $A$. Compactness of $(\e^{tA})_{t\ge 0}$ follows from the compact embedding of $D(A)$ into $C[0,1]$, a direct consequence of the theorem of Ascoli--Arzel\`a. 
 
 To conclude the proof, observe that
\[
V\hookrightarrow C[0,1]\hookrightarrow L^2(0,1).
\]
Accordingly, by~\cite[Prop.~IV.2.17]{EngNag00} the spectra of $A$ and $\tilde{A}$ coincide, and in particular we obtain $s(A)=0$ since $s(\tilde{A})=0$ by Lemma~\ref{lem_generA}.(2).
\end{proof}

We have already observed that $f_0$ and $f_1,$ defined in~\eqref{def:f0} and~\eqref{def:f1}, treated as members of $L^2(0,1)$, form an orthogonal basis of the null space of $\tilde A$
 with $\|f_0\|_{L^2} = 1$ and $\|f_1\|_{L^2}^2 = 12.$ Also, 
 \[
 F_1\fod = \int_0^1 \left(x - \frac 12\right) f(x) \ud x = \frac 1{12} (f|f_1)_{L^2},
 \] so that, for $f \in L^2(0,1)$, the operator $P$ defined in \eqref{pe} coincides with the projection on the null space of $\tilde{A}$.


In view of the above results, we are finally in the position to prove that the semigroup $(\e^{tA})_{t\ge 0}$ converges, exponentially and in norm, towards a rank-2 projection. In particular, it is bounded.

\begin{thm}\label{equilibr-nonneum}
The semigroup $(\e^{tA})_{t\ge 0}$ generated by $A$ converges towards a rank-2 projection. More precisely, for all $\epsilon>0$ such that $-\epsilon$ is strictly larger than the second largest eigenvalue of $A$  there is $M=M(\epsilon) $ such that
\begin{equation}\label{gw}
\| \e^{tA}-P\|_{\mathcal L(C[0,1])} \le M \e^{-\epsilon t},\qquad t\ge 0,
\end{equation}
where $P$ is defined in \eqref{pe}.

\end{thm}

\begin{proof}
It follows from the spectral theorem for self-adjoint operators that each eigenvalue of $\tilde{A}$ is  a simple pole of the resolvent of $\tilde{A}$. Now, in view of the results in~\cite[\S~IV.2.b]{EngNag00} the spectra of $A$ and $\tilde{A}$ agree and the resolvent of $A$ is simply the restriction to $C[0,1]$ of the resolvent of $\tilde{A}$. Furthermore, since the resolvent $(\cdot-\tilde{A})^{-1}$ of $\tilde{A}$ has at each eigenvalue $\lambda$ a first order pole, i.e., the spectral projection associated with $\lambda$ has 1-dimensional range, see \cite[p.~246]{EngNag00}, so does the resolvent of $A$.  In other words,   each eigenvalue of $A$ is  a simple pole of the resolvent of $A$.
In particular, 0 is a simple pole of $A^{-1}$ whose residue is exactly the rank-2 projection on ${\rm ker}(A)$, i.e., the operator $P$. 
Now, a special case of the assertion of~\cite[Cor.~V.3.3]{EngNag00} states that because $0$ is a dominant eigenvalue (as in our case, since $0$ is the spectral bound and an eigenvalue of $A$) of the generator  of an eventually compact semigroup (ours is even immediately compact, 	by Lemma~\ref{lemmengelA}), and because 0 is also a first order pole of the resolvent of $A$, the semigroup converges towards the associated residue $P$ in the way described in equation \eqref{gw}. 
 \end{proof}

Let us observe that the semigroup generated by $\tilde{A}$ can be shown to be non-positive -- this can be done by the Beurling--Deny criterion. This explains why our entire proof relies solely upon properties of self-adjoint semigroups.

\bibliographystyle{plain}
\bibliography{bib}
\end{document}